\documentclass[11pt]{article}
\usepackage{amssymb}
\usepackage{amsmath}
\usepackage{amsthm}
\usepackage{latexsym}
\usepackage[english]{babel}
\usepackage{appendix}

\usepackage{hyperref}

\newtheorem{theorem}{Theorem}[section]

\newtheorem{lemma}{Lemma}[section]

\newtheorem{corollary}{Corollary}[section]

\def \de {\partial}

\def \O {\Omega}
\def \la {\langle}
\def \ra {\rangle}

\def \R {\mathbb{R}}

\def \LH {\Delta_\H}
\def \H {\mathbb{H}}
\def \Fa {F_{\alpha}}
\def \divH {\textnormal{div}_{H}}
\def \a {\alpha}

\newcommand{\Hn}{{\mathbb{H}^n}}

\begin{document}

\title{Overdetermined problems for gauge balls\\ in the Heisenberg group}

\author{Vittorio Martino$^{(1)}$ \& Giulio Tralli$^{(2)}$}
\addtocounter{footnote}{1}
\footnotetext{Dipartimento di Matematica, Universit\`a di Bologna, piazza di Porta S.Donato 5, 40126 Bologna, Italy. E-mail address:
{\tt{vittorio.martino3@unibo.it}}}
\addtocounter{footnote}{1}
\footnotetext{Dipartimento d'Ingegneria Civile e Ambientale (DICEA), Universit\`a di Padova, Via Marzolo 9, 35131 Padova, Italy. E-mail address: {\tt{giulio.tralli@unipd.it}}}

\date{}
\maketitle

\vspace{5mm}

{\noindent\bf Abstract} {\small In this paper we aim at characterizing the gauge balls in the Heisenberg group $\Hn$ as the only domains where suitable overdetermined problems of Serrin type can be solved. We discuss a one parameter family of overdetermined problems where both the source functions and the Neumann-like data are non-constant and they are related to the geometry of the underlying setting. The uniqueness results are established in the class of domains in $\Hn$ having partial symmetries of cylindrical type for any $n\geq 1$, and they are sharper in the lowest dimensional cases of $\H^1$ and $\H^2$ where we can respectively treat domains with $S^1$ and $S^1\times S^1$ invariances.}

\vspace{5mm}


\vspace{4mm}


\vspace{4mm}

\section{Introduction}

\noindent The main interest of the present work is to establish symmetry results for solutions of overdetermined problems in $\R^{2n+1}$ (with $n\geq 1$, and generic point $\xi=(x,t)\in\R^{2n}\times\R$) such as
\begin{equation}\label{modeloverdet}
\begin{cases}
\Delta_x u + 4|x|^2\de^2_{tt}u + 4\left\langle Jx, \nabla_x \de_t u\right\rangle   =(2n+2)\frac{|x|^2}{\sqrt{|x|^4+t^2}} \quad\text{ in } \O,  \\
u=0 \qquad\qquad\qquad\qquad\qquad\qquad\qquad\qquad\qquad\quad\quad\,\,\,\,\,\,\text{ on }\de \O,  \\
\left(\sum_{j=1}^{2n}(\de_{x_j}u+2(Jx)_j\de_tu)^2\right)^{\frac{1}{2}}= c \frac{|x|}{\left(|x|^4+t^2\right)^{\frac{1}{4}}} \qquad\qquad\,\,\,\,\,\text{ on }\de \O. 
\end{cases}
\end{equation}
Here, and in what follows, $c$ is a positive constant, $\la \cdot, \cdot\ra$ and $|\cdot|$ denote the standard inner product and norm in $\R^{2n}$, and $J$ stands for the following standard symplectic matrix in $\R^{2n}$
\begin{equation}\label{defJ}
J=\begin{pmatrix}
0 & -I_{n} \\
I_{n} & 0
\end{pmatrix}.
\end{equation}
Despite being a far-reaching (and maybe weird-looking) generalization of the classical overdetermined problem for the torsion function studied by Serrin in \cite{Ser}, the problem \eqref{modeloverdet} should first be looked with the eyes and the notations proper of the Heisenberg group. As we postpone to Subsection \ref{motiv} the description of our motivations for studying \eqref{modeloverdet}, we proceed in recalling the geometric setting and stating the main results.

\vskip 0.4cm

\noindent The Heisenberg group $\Hn$ can be identified with $\R^{2n+1}$ once we fix the group law
\begin{equation*}
\xi\circ \xi'=(x,t)\circ(x',t')=(x+x',t+t'+2\la Jx,x'\ra), \qquad \mbox{ for }  \xi,\xi' \in\Hn.
\end{equation*}
The $1$-parameter family of group homomorphisms $\{\delta_{\lambda}\}_{\lambda>0}$ defined as
\begin{equation*}
\delta_{\lambda}:\Hn\to\Hn, \qquad \delta_{\lambda}(x,t)=(\lambda x,\lambda^2 t)
\end{equation*}
plays the role of the homogeneous dilations, and it displays the non-isotropy of the setting under discussion. Such non-isotropy is also reflected by the so-called gauge function
\begin{equation}\label{defgauge}
\rho(\xi) =\left(|x|^4+t^2\right)^{\frac{1}{4}}, \qquad \xi \in \Hn.
\end{equation}
The function $\rho$ defined in \eqref{defgauge} sits at the center stage of our analysis. On one hand, it defines a norm which is $1$-homogeneous with respect to $\delta_\lambda$, and it endows $\Hn$ with the structure of a metric space where the metric balls centered at $\xi_0\in\Hn$ and radius $R>0$ are given by
\begin{equation}\label{defballs}
B_R(\xi_0)=\left\{ \xi \in \Hn \; : \; \rho(\xi_0^{-1}\circ\xi) < R \right\}.
\end{equation}
The balls $B_R(\xi_0)$ in \eqref{defballs} are called \emph{gauge balls} or \emph{Kor\'anyi balls} in the literature, as they first appeared in \cite{KV} in the study of singular integrals on homogeneous spaces. On the other hand, the gauge function is very much related with the sub-elliptic partial differential operator we work with. The canonical basis of left invariant vector fields on $\Hn$ is given by
\begin{equation*}
X_j=\frac{\partial}{\partial x_j}+2(Jx)_j\frac{\partial}{\partial t},\quad T=\frac{\partial}{\partial t}, \quad j=1,\dots,2n.
\end{equation*}
The vector fields $\{X_1,\ldots,X_{2n}\}$ are called horizontal: they are $1$-homogeneous with respect to $\delta_\lambda$ and they span the kernel of the contact form of the underlying CR structure. The second order partial differential operator defined by
\begin{equation}\label{LHdef}
\LH=\sum_{j=1}^{2n} X_j^2
\end{equation}
is called Heisenberg subLaplacian. It is a direct computation to recognize that
\begin{equation}\label{explicitDelta}
\LH=\Delta_x u + 4|x|^2\de^2_{tt}u + 4\left\langle Jx, \nabla_x \de_t u\right\rangle,
\end{equation}
namely the operator appearing in \eqref{modeloverdet}. From \eqref{LHdef} and the properties of the vector fields, it is clear that $\LH$ is left-invariant and $\delta_\lambda$-homogeneous of degree $2$. It is also hypoelliptic (even if never elliptic, at any point) thanks to the non-commutation relation $[X_j,X_i]=4J_{ij} T$. Since \cite[Theorem 2]{Fo73} it is well-known that the fundamental solution of $\Delta_{\H}$ is gauge-symmetric. More precisely, for any $\xi_0\in\Hn$ the function
\begin{equation}\label{fond}
\xi\mapsto \frac{\beta}{\rho^{Q-2}(\xi_0^{-1}\circ \xi)}\quad \mbox{is the fundamental solution of $\Delta_{\Hn}$ with pole at $\xi_0$,}
\end{equation}
where $\beta>0$ is a renormalizing constant. Here, as well as in what follows, we have indicated by
$$
Q=2n+2
$$
the homogeneous dimension of the group. Denoting by $D_H u$ the so-called horizontal gradient of a given function $u$ which is described by the $\R^{2n}$-vector
\begin{equation*}
D_H u=\left(X_1 u,\ldots, X_{2n} u\right),
\end{equation*}
a straightforward computation which exploits the skew-symmetry of the matrix $J$ ensures that
\begin{equation}\label{dhrho}
|D_H \rho (\xi)|=\frac{|x|}{\rho(\xi)}\qquad\mbox{ for any }\xi\in\Hn,\,\, \xi\not=0.
\end{equation}
The last notation we need concerns the nonnegative function
\begin{align}\label{deFa}
\Fa(\xi)&=|x|^2 \rho^{\alpha-4}(\xi)\\
&=|D_H \rho (\xi)|^2 \rho^{\alpha-2}(\xi),\notag
\end{align}
where $\alpha$ is a positive constant which will be settled in a moment in some range contained in $(0,4]$. We are thus ready to declare the overdetermined system depending on the parameter $\alpha$ which constitutes the core of our work
\begin{equation}\label{eq: overdet}
\begin{cases}
\LH u =(Q+\alpha-2)\Fa \quad\,\,\,\text{ in } \O,  \\
u=0 \qquad\qquad\qquad\qquad\,\,\,\,\,\text{ on }\de \O,  \\
|D_H u|= c\, \Fa^{\frac{1}{2}}\qquad\quad\,\,\,\,\,\,\,\,\,\,\,\,\,\, \text{ on }\de \O.  \\
\end{cases}
\end{equation}
We postpone to Section \ref{sec2} the discussion about the precise meaning of solution $u$ to the system \eqref{eq: overdet}, and for the assumptions on the class of bounded opens sets $\Omega$ where we seek to solve \eqref{eq: overdet}. We explicitly note that, according to our notations, the case $\alpha=2$ corresponds to the desired system \eqref{modeloverdet}. The case $\alpha=2$ occupies a special position also for another reason: as the function $\Fa$ is $\delta_\lambda$-homogeneous of degree $\alpha-2$ (see \eqref{dhrho}), in the $(\alpha=2)-$system \eqref{modeloverdet} both the source function and the Neumann-like datum are $\delta_\lambda$-homogeneous of degree $0$ (just like the constants) and in this respect \eqref{modeloverdet} represents a true generalization of the classical Serrin overdetermined problem for the torsion function.

\vskip 0.4cm

\noindent For any $\alpha>0$, a direct computation based on \eqref{fond}-\eqref{dhrho} shows that in $\Omega=B_{R}(0)$ for any $R>0$ the unique solution $u_\alpha$ (in the sense of Section \ref{sec2}) to \eqref{eq: overdet} is given by
\begin{equation}\label{ua}
u_\alpha(\xi)= \frac{\rho^{\alpha}(\xi) - R^\alpha}{\alpha}, \qquad\mbox{ and in such case }c=R^{\frac{\alpha}{2}}.
\end{equation}
In the particular case $\alpha=4$ the system \eqref{eq: overdet} inherits an extra degree of freedom as the function $F_{4}(\xi)=|x|^2$ is independent of the $t$-variable: this fact is responsible for the existence of gauge-symmetric solutions in $\Omega=B_{R}((0,t_0))$ for any $t_0\in\R$ and $R>0$ which are given by $u_4(\xi)=\frac{1}{4}(|x|^4+(t-t_0)^2 - R^4)$. Similar symmetry problems in $\Hn$ with weights depending just on $|x|$ (like $F_4$) have been recently addressed in a companion paper \cite{GMT}\footnote{the content of the present paper was in fact announced in \cite{GMT} in July 2022 (see reference [28] therein)}, where we proved that the gauge spheres $\de B_{R}((0,t_0))$ are the unique smooth closed hypersurfaces in $\Hn$ with horizontal mean curvature proportional to $|x|$ in the class of domains which are cylindrically symmetric with respect to the $t$-axis (see, e.g., \cite[Corollary 3.6]{GMT}). Here, in contrast with \cite{GMT}, we allow the weights $\Fa (x,t)$ to be possibly $t$-dependent via the presence of the function $\rho^{\alpha-4}(x,t)$.

\vskip 0.4cm

\noindent The aim of the present work is to prove that the gauge balls are the unique domains $\Omega$ in a class of competitor sets with cylindrical-type symmetries where \eqref{eq: overdet} admits a solution for some positive constant $c$ and some positive parameter $\alpha$ in a suitable range. The requirement for the domains to have an a-priori symmetry assumption is not new in the literature: due to the known difficulties in performing symmetrizations, moving plane techniques, and Bochner-type identities in sub-Riemannian settings, this type of axially symmetric assumption already appeared in a number of notable contributions by many authors such as \cite{JL87, GL90, BiPr, HL, RiRo, Mo, MTj} (to the best of our knowledge, the only symmetry results available in $\Hn$ without hypotheses concerning partial symmetries are the important classifications of global solutions/minima in \cite{JL88, BFM, FL, MaOu} strongly based on the underlying conformal invariances). More precisely, we shall say that a bounded open set $\Omega\subset\Hn$ is cylindrically symmetric (with respect to the $t$-axis) if
\begin{equation}\label{cylsim}
(x,t)\in\Omega \quad \Longrightarrow \quad (x', t)\in \Omega \quad\mbox{ for every $x'\in \R^{2n}$ with $|x'|=|x|$.}
\end{equation}
In this respect, our main result is the following

\begin{theorem}\label{Tuno}
Fix $n\geq 1$, and $\alpha\in (\alpha_n,4]$ where $\alpha_n = \frac{3}{4}$ for $n\geq 2$ and $\alpha_1=2$. Fix also $c>0$. Let $\Omega$ be a subset of $\Hn$ satisfying the assumptions in Section \ref{sec2}, and suppose that $\Omega$ is cylindrically symmetric. If there exists a solution $u$ to \eqref{eq: overdet}, then $\Omega$ is a gauge ball of radius $R=c^{\frac{2}{\alpha}}$.
\end{theorem}

For a refined statement of Theorem \ref{Tuno}, we refer the reader to Theorem \ref{Tunoalpha}, Corollary \ref{Cuno}, and Theorem \ref{Tuno4} below. In Theorem \ref{Tunoalpha} we deal with the characterization of the gauge balls centered at the origin $B_R(0)$ for every $\alpha$ in $(0,4)$ under an extra-assumption concerning the local behaviour of an auxiliary function (see Corollary \ref{Cuno} for the case $\alpha\in (\alpha_n,4)$), whereas in Theorem \ref{Tuno4} we deal with the characterization of the gauge balls centered on the $t$-axis $B_R((0,t_0))$ in case\footnote{concerning only $\alpha=4$, the content of Theorem \ref{Tuno4} is also discussed in a very recent preprint \cite{GVpreprint} where the authors provide a number of conjectures via some interesting and formal analogies with classical symmetry problems in $\R^N$ for possibly non-integers $N$} $\alpha=4$. The strategy we adopt for the proof of Theorem \ref{Tuno} follows the one pioneered by Weinberger in \cite{wei} for the classical Serrin overdetermined problem which is based on integral identities, maximum principles, and sharp matrix inequalities of Newton-type. The main difference with the elegant argument in \cite{wei}, as well as the main novelty of the present paper at a technical level, is that the scope of the proof is not to ensure that some Hessian matrix is a multiple of the identity (as it is not even true for the candidate solution $u_\alpha$ in \eqref{ua}). To specify better this difficulty, we should stress again the difference between the simpler situation $\alpha=4$ (for which the candidate solution $u_4$ possesses, at least in the cylindrical coordinates $(\sigma,t)\in \R^+\times\R$ with $\sigma=|x|^2$, an Hessian matrix which is equal to the identity) and the case $0<\alpha<4$ (for which the function $(\sigma^2+t^2)^{\frac{\alpha}{4}}$ is radial but of course not a quadratic polynomial). Therefore, in order to detect the gauge-symmetry of the solutions, we are going to construct a special function $v$ (a so-called P-function) which is singular on the $t$-axis and with the property that $\LH v$ measures the distance from the solution $u$ to the candidate solution $u_\alpha$ in terms of a rather long list of second-order and first-order quantities. Having this in mind, it is worth noticing that we manage to show the validity of Theorem \ref{Tuno} for a non-trivial interval of parameters $\alpha$ with a common approach \`a la Weinberger as it is known that the Weinberger's proof in $\R^N$ is very rigid with respect to the choice of the source function and the Neumann data (as far as we know, in the elliptic framework the interesting uniqueness results for overdetermined problems with position-dependent data are achieved with approaches very different than the one in \cite{wei}, see, e.g., \cite{BH, Br, Ono, DvEPs}). The crucial property that comes to rescue us is the fact that the a-priori symmetry assumptions allow us to work in a low-dimensional space and the approaches via integral identities and sharp inequalities are less rigid in $2$-dimensions (this phenomenon seems unnoticed in the literature for the Weinberger's proof, but it appears for other symmetry results in the plane such as \cite{CrSp, Lions}).

\vskip 0.4cm

\noindent With the intention of enlarging the class of $\Omega$ for which we have a uniqueness result, we are going to work with domains with some a-priori symmetry assumptions which are weaker than \eqref{cylsim}. We shall say that a bounded open set $\Omega\subset\Hn$ is toric symmetric (with respect to the symplectic matrix $J$ \eqref{defJ}, and the $t$-axis) if
\begin{align}\label{torsim}
(x,t)\in\Omega \quad \Longrightarrow \quad (x', t)\in \Omega \quad&\mbox{ for every $x'\in \R^{2n}$ such that }\\
&\mbox{$(x')^2_k+(x')^2_{n+k}=x^2_k+x^2_{n+k}$ with $k\in\{1,\ldots,n\}$.}\notag
\end{align}
For $n=1$, \eqref{torsim} coincides with the $S^1$-invariance described in \eqref{cylsim}. However, for $n\geq 2$, \eqref{torsim} constitutes a weaker requirement than \eqref{cylsim}: \eqref{torsim} is equivalent of saying that the sections $\{x\in\R^{2n}: (x,t)\in\Omega\}$ are Reinhardt domains in $\mathbb{C}^{n}=(\R^{2n},J)$ (and not just balls as for \eqref{cylsim}). In Section \ref{sec3} we are going to introduce the P-function $v$ associated to a solution of the overdetermined system \eqref{eq: overdet}, and we will see that $v$ is going to be effective as a detector of the gauge-symmetry not just in the case of cylindrically symmetric sets for any $n\geq 1$ but also in the case of toric symmetric sets in $\mathbb{H}^2$. We summarize the latter result in the following theorem, which we feel is close (at least in spirit) with the Aleksandrov-type theorem for the Levi curvature in the class of Reinhardt and circular domains in $\mathbb{C}^2$ obtained in \cite{HL, MTj}.  

\begin{theorem}\label{Tdue}
Fix $n=2$, $\alpha\in (\frac{3}{4},4]$, and $c>0$. Let $\Omega$ be a subset of $\mathbb{H}^2$ satisfying the assumptions in Section \ref{sec2}, and suppose that $\Omega$ is toric symmetric. If there exists a solution $u$ to \eqref{eq: overdet}, then $\Omega$ is a gauge ball of radius $R=c^{\frac{2}{\alpha}}$.
\end{theorem}

As before, we refer the reader to Theorem \ref{Tduealpha}, Corollary \ref{Cdue}, and Theorem \ref{Tdue4} below for a refined statement of Theorem \ref{Tdue} accordingly to the values of $\alpha$.

\vskip 0.4cm

\noindent The paper is organized as follows. In Section \ref{sec2} we set the required hypotheses for the class of competitors $\Omega$, we settle the definition of solution to the system \eqref{eq: overdet}, and we derive some preliminary integral identities. In Section \ref{sec3} we introduce the P-function $v$ and we show a crucial pointwise identity involving $\LH v$ in case of domains with toric symmetries. We proceed by discussing both the subharmonicity and local integrability properties of $v$ with respect to the values of the parameter $\alpha$.  In Section \ref{sec4} we provide the complete details for the proof of Theorems \ref{Tuno} and \ref{Tdue}.

\subsection{Motivations}\label{motiv}

Our interest for studying the specific overdetermined problem in \eqref{modeloverdet} comes from a natural inverse problem in potential theory. It is in fact known since a celebrated paper by Gaveau \cite[Th\'eor\`eme 3]{Ga77} (see also the general treatment in \cite[Theorems 5.5.4 and Theorem 5.6.1]{BLU}) the validity of Gauss-Koebe type theorems in $\Hn$ stating that the pointwise value of every $\LH$-harmonic function can be represented as a weighted average of the values of $h$ on gauge balls $B_R$ and on gauge spheres $\de B_R$. More precisely, if we have a solution $h$ to $\Delta_{\Hn}h=0$ in $B_R(0)$ (which is continuous up to the closure of $B_R(0)$) we then have that
\begin{align}\label{mean}
h(0)&=\frac{Q(Q-2)\beta}{R^{Q}}\int_{ B_R(0)} h(\xi) |D_H\rho(\xi)|^2 d\xi\\
&=\frac{(Q-2)\beta}{R^{Q-1}}\int_{\de B_R(0)} h(\xi) \frac{|D_H\rho(\xi)|^2}{|D \rho(\xi)|} d\sigma(\xi).\notag
\end{align}
Here, and in what follows, we have indicated by $d\xi$ the Lebesgue measure, by $d\sigma (\xi)$ the surface measure for smooth hypersurfaces, and by $D$ the Euclidean gradient. The positive constant $\beta$ is the same constant appearing in \eqref{fond}. As a matter of fact, the validity of \eqref{mean} is mainly due to \eqref{fond} (being the gauge balls the superlevel sets of the fundamental solution) and to the variational structure of $\LH$ (being the horizontal vector fields $X_j$ divergence-free with respect to the Lebesgue measure), see in this respect also \cite{BoLa}. When we refer to inverse problems we intend the characterization of the domains where mean value properties might hold, see the very insightful survey \cite{NV}. Lanconelli proved in \cite{L13} a Kuran-type theorem by characterizing the gauge balls in $\Hn$ as the only domains where the pointwise value of $\LH$-harmonic functions coincides with the weighted solid-average displayed in \eqref{mean} and given by the weight $|D_H\rho(\xi)|^2$ (the proof in \cite{L13} relies on the identification of the fundamental solution from equilibrium potentials constructed via group convolutions, see also \cite{ASZ} for the case of classical Newtonian potentials). On the other hand, it is known (see, e.g., \cite{Sh}) that the classical Serrin overdetermined problem
$$
\begin{cases}
\Delta u =N \quad\,\,\,\text{ in } D,  \\
u=0 \qquad\,\,\,\,\,\text{ on }\de D,  \\
|D u|= c\quad\,\,\, \text{ on }\de D  \\
\end{cases}
$$
is related with the characterizations of domains $D$ in $\R^N$ for which the solid average and the surface average of harmonic functions do coincide. We would like to convince the reader that the same applies for the weighted averages in \eqref{mean} and our overdetermined system \eqref{modeloverdet} (namely the case $\alpha=2$ of \eqref{eq: overdet}) which we can rewrite in our notations as
\begin{equation}\label{quantebello}
\begin{cases}
\LH u =Q\, |D_H\rho|^2 \quad\,\,\,\,\text{ in } \Omega,  \\
u=0 \qquad\,\,\,\,\,\,\,\,\,\,\,\quad\quad\,\,\,\,\text{ on }\de \Omega,  \\
|D_H u|= c\,|D_H\rho| \quad\,\,\,\, \text{ on }\de \Omega.  \\
\end{cases}
\end{equation}
The fact that the gauge-function $\rho$ does not satisfy the relevant eikonal equation (i.e. $|D_H\rho(\xi)|=\frac{|x|}{\rho(\xi)}\not\equiv 1$) appears then as the subtle difference between the two systems.

\vskip 0.4cm

\noindent In order to do so, we need a few more notations. The horizontal divergence of a vector field $(V_1, \ldots, V_{2n})$ is given by
\begin{equation*}
\divH (V)=\sum_{j=1}^{2n} X_j ( V_j ),
\end{equation*}
so that one has $\LH u = \divH (D_H u)$. For a given bounded open set $\O\subseteq\Hn$ with smooth boundary, we say that $f$ is a defining function for $\O$ if $f$ is smooth in a neighborhood of $\de\Omega$ and
\begin{equation*}
\O=\{\xi \in \Hn\; : \; f(\xi)<0\}, \qquad \de\O=\{\xi \in \Hn\; : \; f(\xi)=0\}, \qquad Df\neq 0 \; \text{on} \; \de\O .
\end{equation*}
We denote by $\nu$ the Euclidean outer normal to $\de\O$, and we define the horizontal outer normal to $\de\O$ as follows
\begin{equation*}
\nu^H(\xi):=\frac{D_H f(\xi)}{|D_H f(\xi)|}, \quad \text{for any $\xi\in \de\O$ such that $D_H f(\xi)\neq 0$.}
\end{equation*}
A point $\xi\in \de\O$, such that $D_H f(\xi)= 0$, is called characteristic. In the literature there is a well-established notion of horizontal perimeter measure of $\de\Omega$ (see, e.g., \cite{CDPT}) which can be defined as $ d\sigma_H=\frac{|D_H f|}{|D f|} d\sigma$. In particular, for smooth enough functions $v$, one has 
\begin{equation*}
\int_\O X_j(v)(\xi) \; d\xi= \int_{\de\O} v(\xi) \; \frac{X_j f(\xi)}{|Df(\xi)|} \; d\sigma(\xi) =\int_{\de\O} v(\xi) \; \nu^H_j(\xi) \; d\sigma_H(\xi),\quad j=1,\dots,2n.
\end{equation*}
Hence, since $f(\xi)=\rho(\xi)-R$ works as a defining function for $\de B_R(0)$, we can rewrite the surface average in \eqref{mean} as
$$\frac{(Q-2)\beta}{R^{Q-1}}\int_{\de B_R(0)} h(\xi) \frac{|D_H\rho(\xi)|^2}{|D \rho(\xi)|} d\sigma(\xi)=\frac{(Q-2)\beta}{R^{Q-1}}\int_{\de B_R(0)} h(\xi) |D_H\rho(\xi)| d\sigma_H(\xi).$$
Denoting (for $\Omega$ such that $0\not\in \de \Omega$)
$$
A_{solid}(\Omega)=\int_\Omega |D_H\rho(\xi)|^2 d\xi\quad\mbox{ and }\quad A_{surface}(\de\Omega)=\int_{\de\Omega} |D_H\rho(\xi)| d\sigma_H(\xi)
$$
and using the fact that $1$ is $\LH$-harmonic, it is clear from \eqref{mean} that
$$
\frac{1}{A_{solid}(B_R(0))} \int_{B_R(0)} h(\xi) |D_H\rho(\xi)|^2 d\xi = \frac{1}{A_{surface}(\de B_R(0))} \int_{\de B_R(0)} h(\xi) |D_H\rho(\xi)| d\sigma_H(\xi)
$$
for any $\LH$-harmonic function $h$ in $B_R(0)$ (continuous up to the boundary). Let us now consider a solution $u$ to our overdetermined system \eqref{quantebello} in $\Omega$ (assume for now that everything is smooth enough in order to be well-defined, and $0\in\Omega$) and consider any $\LH$-harmonic function $h$ in $\Omega$ (smooth up to the boundary), we can deduce that
\begin{align}\label{perognih}
&\int_{\Omega} h(\xi) |D_H\rho(\xi)|^2 d\xi = \frac{1}{Q} \int_{\Omega} h(\xi) \LH u(\xi) d\xi = \frac{1}{Q} \int_{\Omega} \left( h(\xi) \LH u(\xi) - u(\xi) \LH h(\xi)\right) d\xi \notag\\
&=\frac{1}{Q} \int_{\Omega} \divH \left( h D_H u - u D_H h\right)(\xi) d\xi = \frac{1}{Q} \int_{\de\Omega} h(\xi) |D_H u(\xi)| d\sigma_H(\xi)\notag\\
&=\frac{c}{Q} \int_{\de\Omega} h(\xi) |D_H \rho(\xi)| d\sigma_H(\xi)
\end{align}
where we exploited the fact that $u$ is also a defining function for $\Omega$ (see \eqref{defining} below for more details). In particular, 
by plugging $h\equiv 1$ in \eqref{perognih} it yields
$$
\frac{c}{Q}=\frac{A_{solid}(\Omega)}{A_{surface}(\de\Omega)}.
$$
Substituting in \eqref{perognih} the previous relationship, we finally obtain
$$
\frac{1}{A_{solid}(\Omega)} \int_{\Omega} h(\xi) |D_H\rho(\xi)|^2 d\xi = \frac{1}{A_{surface}(\de \Omega)} \int_{\de \Omega} h(\xi) |D_H\rho(\xi)| d\sigma_H(\xi)
$$
for any smooth $\Omega$ admitting a solution to our overdetermined system.

\section{Preliminaries: assumptions and integral identities}\label{sec2}

Let us now define, for $\alpha >0$, our meaning for competitor sets $\Omega$ and for solutions to the overdetermined system \eqref{eq: overdet}. We consider an open, bounded and connected set $\Omega\subset \Hn$ with smooth boundary, we assume that around characteristic points of the boundary (defined in the previous section) the set $\O$ has interior and exterior tangent gauge-balls, and we ask that the unique weak solutions to
\begin{equation}\label{weak}\begin{cases}
\LH u =(Q+\alpha-2)\Fa \quad\,\,\,\text{ in } \O,  \\
u=0 \qquad\qquad\qquad\qquad\,\,\,\,\,\text{ on }\de \O
\end{cases}\end{equation}
are also smooth in a neighborhood of the boundary of $\Omega$. For such $\Omega$, we say that $u$ solves \eqref{eq: overdet}, if $u$ is the unique weak solution to \eqref{weak} and satisfies
\begin{equation}\label{extra}
|D_H u(\xi)|= c\, \Fa^{\frac{1}{2}}(\xi)\quad \mbox{ for }\xi\in\de\Omega
\end{equation}
for some positive constant $c$. We recall that a weak solution $u$ to \eqref{weak} is a function in the Folland-Stein space $H^1_0(\Omega,X)$ (of functions in $L^2$ with horizontal gradient in $L^2$) such that
$$
\int_\Omega \left\langle D_H u, D_H\psi\right\rangle + (Q+\a-2)\int_\Omega \Fa \psi = 0 \quad \mbox{ for all }\psi\in C_0^\infty(\Omega).
$$
To summarize the assumptions and notations discussed here, we shall say that
\begin{equation}\label{competitor}
\Omega \mbox{ is a competitor set}
\end{equation}
and
\begin{equation}\label{solution}
u \mbox{ is solution to \eqref{eq: overdet} in }\Omega.
\end{equation}
It is safe to remark that the gauge balls are competitor sets and the function $u_{\alpha}$ in \eqref{ua} is solution to \eqref{eq: overdet} in a gauge ball for any $\alpha>0$.

\vskip 0.4cm

Some comments are in order. Being the function $\Fa$ non-smooth at $\xi=0$ (whenever $\alpha< 4$), we assume $0 \in \Omega$ (whenever $\alpha< 4$) in order to fix the ideas. We can then understand the assumptions by splitting the discussion into interior and boundary issues. Since
$$
\Fa\in C^{\infty}(\Omega\smallsetminus \{0\})\quad \mbox{and}\quad
\Fa\in \begin{cases}
L^\infty(\Omega) \quad\,\,\,\text{ if } \alpha\geq 2,  \\
L^p(\Omega) \quad\,\,\,\text{ for some }p>\frac{Q}{2}\mbox{ if }0<\alpha<2  \\
\end{cases}
$$
which is a consequence of the boundedness of the term $|D_H\rho|$ in \eqref{deFa}, it is known (see, e.g., \cite{BLU, U}) that weak solutions to \eqref{weak} satisfy $u\in C^{\infty}(\Omega\smallsetminus \{0\})\cap C(\Omega)$ (actually more than just continuous). Moreover, as the boundary data are smooth, from the results \cite{KN, Juno} we know that $u$ is smooth up to every boundary point which is not characteristic. At characteristic boundary points the solutions exhibit a more delicate behaviour (in the sense that they might be non-regular for smooth $\de\Omega$) and sharp geometric conditions which ensure regularity properties are not fully understood yet: we refer to \cite{Jdue} for sufficient conditions for regularity of first and second derivatives at characteristic boundary points. On the other hand, the situation concerning the continuity and the oscillation of the solution at such points is more clear. The exterior gauge-ball condition ensures always the H\"older-continuity at the boundary (see e.g. \cite{LU97}, and also \cite{U} for a Moser-iteration up to the boundary in more general settings). Since $(Q+\a-2)\Fa \geq 0$ and $u\in C(\overline{\Omega})$, we can exploit the weak maximum principle to infer that
\begin{equation}\label{sign}
u<0 \quad \mbox{in }\Omega\qquad\mbox{ and }\qquad u\equiv 0 \quad\mbox{ in a pointwise sense on }\de\Omega.
\end{equation}
The assumption concerning the interior gauge-ball condition at the characteristic point yields also the validity of the so-called Hopf-lemma at every boundary point, i.e. a boundary comparison principle (see \cite{BiCu, MTh}). In particular, since $u$ is assumed to be regular and by \eqref{extra} the characteristic points might arise only on $\{|x|=0\}$ where the normal has to be parallel to the the $t$-axis, we have
$$
|Du|>0\quad\mbox{ on }\de\Omega.
$$
The previous property combined with \eqref{sign} is telling us that under our assumptions we have that
\begin{equation}\label{defining}
u\mbox{ is a defining function for }\Omega.
\end{equation}

\vskip 0.4cm

Let us now write some pointwise relationships concerning the function $\Fa$, which will be useful in the sequel: for any $\xi\not= 0$ we have
\begin{align}\label{derfa}
&D_H \Fa (\xi)= 2x \rho^{\a-4}(\xi) + (\a-4)|x|^2  \rho^{\a-8}(\xi) \left( x |x|^2  + (Jx)t \right),\\
&|D_H \Fa (\xi)|^2= 4 |x|^2 \rho^{2\a-8}(\xi) + \a(\a-4)|x|^6  \rho^{2\a-12}(\xi),\notag\\
&T \Fa (\xi)= \frac{1}{2} (\a-4)|x|^2  \rho^{\a-8}(\xi) t ,	\notag\\
& \LH \Fa (\xi)= 2(Q-2) \rho^{\a-4}(\xi) + (\a-4)(Q+\a-2) |x|^4 \rho^{\a-8}(\xi).\notag
\end{align}
The following crucial integral identity concerns the interplay between the function $\Fa$ and the solutions $u$ of the overdetermined systems under discussion: we stress here the double role of $\Fa$ as a source function and as a Neumann datum. 

\begin{lemma} \label{pohazev}
Fix $c, \alpha>0$. Let $u$ be a solution of \eqref{eq: overdet} in $\Omega$ according to \eqref{competitor}-\eqref{solution}. Then we have
\begin{equation}\label{eq: pohazev}
(Q+2\alpha-2)\int_\O u\; \Fa \;  d\xi=-c^2\int_{\O} \Fa  \;d\xi.
\end{equation}
\end{lemma}

\begin{proof}
The proof of \eqref{eq: pohazev} is based on Pohozaev-type identities, which were first established in this context in \cite{GL92}. To this aim, we in fact consider as in \cite{GL92} the vector field generating the group dilations
$$Z=\sum_{j=1}^{2n} x_j \frac{\partial}{\partial x_j} + 2t\frac{\partial}{\partial t}=\sum_{j=1}^{2n} x_j X_j + 2t T.$$
A direct computation shows that
\begin{equation*}
\LH(Z(u))=2\LH u +Z(\LH u)=\alpha(Q+\a-2) \Fa \quad\mbox{ almost everywhere in }\Omega,
\end{equation*}
where we exploited the fact that $Z(\Fa)=(\alpha-2)\Fa$ (recall that $\Fa$ is $\delta_\lambda$-homogeneous of degree $\alpha-2$). By means of integration by parts, we then obtain
\begin{align}\label{primadelletre}
&\int_\O \big\{ \alpha(Q+\a-2) u \Fa  - (Q+\a-2) Z(u) \Fa \big\} \; d\xi \\
&=\int_\O \big\{ \LH(Z(u)) u -  Z(u)\LH(u) \big\} \; d\xi =\notag\\
& = \int_{\de\O} \big\{ u\; \la Z(u),\nu^H \ra -  Z(u) \; \la D_H u, \nu^H \ra \big\} \; d\sigma_H=
- \int_{\de\O}  Z(u) \; |D_H u|  \; d\sigma_H\notag
\end{align}
where we used \eqref{defining}. Now, since
$$Z(u) \; |D_H u|=\left(\la x,D_H u \ra + 2t\; Tu \right)|D_H u|=\left(\la x, \nu^H \ra + 2t\; \frac{Tu}{|D_H u|} \right)|D_H u|^2 , $$
and exploiting \eqref{extra}, we get
\begin{equation}\label{secondadelletre}
\int_{\de\O}  Z(u) \; |D_H u|  \; d\sigma_H =
c^2\int_\O \big\{Q \Fa +  Z(\Fa) \big\} \; d\xi =
c^2(Q+\a-2)\int_\O \Fa \; d\xi .
\end{equation}
On the other hand, we also have
\begin{align}\label{terzadelletre}
 (\a-2)\int_{\O}  u  \Fa  \; d\xi = & \int_{\O}  u  Z(\Fa)  \; d\xi=\\
 \int_\O \big\{Z(u\Fa) -  Z(u)\Fa \big\} \; d\xi = & -Q\int_{\O}  u  \Fa  \; d\xi - \int_{\O}  Z(u)  \Fa  \; d\xi .\notag
\end{align}
By putting together the three integral identities \eqref{primadelletre}-\eqref{secondadelletre}-\eqref{terzadelletre} and using $Q+\a-2>0$, we get the desired formula \eqref{eq: pohazev}.
\end{proof}

\section{Weighted P-functions for domains with toric symmetries}\label{sec3}

For any solution $u$ to \eqref{eq: overdet} in a competitor set $\Omega$, we introduce the function
\begin{equation}\label{Pfunction}
v=\frac{1}{\Fa}|D_H u|^2 -\a u .
\end{equation}
It is clear that $v$ is defined in a pointwise sense in $\Omega\smallsetminus \{|x|=0\}$. Moreover, thanks to the Dirichlet and Neumann conditions for $u$, we know that $v$ approaches the constant $c^2$ at $\de\Omega$ (at least at the points of the boundary which are not on the $t$-axis).\\
On the other hand, since
$$
\int_\O |D_H u|^2 = - \int_\O u \LH u = - (Q+\a -2)\int_\O u \Fa,
$$
we obtain from \eqref{eq: pohazev} that
\begin{equation}\label{medianulla}
\int_\O v\; \Fa \;  d\xi=c^2\int_{\O} \Fa \;d\xi.
\end{equation}
This is saying that, with respect to the weight $\Fa$, the function $v$ is in average with its boundary datum $c^2$. In order to claim that $v$ is the so-called P-function we are looking for, we need to show some form of sub-harmonicity properties for $v$. We recall that, in the context of the Heisenberg group, the only Obata-type arguments available are in the whole space $\Hn$ and in this respect the first appearance of such auxiliary functions is in the celebrated paper by Jerison and Lee \cite[Theorem 7.8]{JL87} dealing with cylindrical solutions of the CR-Yamabe equation. Recently, auxiliary functions of rather different nature but displaying similar singularities as $v$ on the $t$-axis appeared in our work \cite{GMT}. The function $v$ deals instead with solutions to the overdetermined systems \eqref{eq: overdet} in bounded domains, and in the next lemma we want to study its behaviour for solutions with toric symmetries.

\begin{lemma}\label{lemmalemma}
Fix $\a\in\R$. Let $u$ be a smooth solution to $\LH u =(Q+\alpha-2)\Fa$ in the open set $\Omega\smallsetminus \{|x|=0\}$ with $\Omega\subset \Hn$ being toric symmetric. Assume that
$$
u(x,t)=U(x^2_1+x^2_{n+1},\ldots,x^2_n+x^2_{2n},t)
$$
for some smooth function $U$, and denote $v=\frac{1}{\Fa}|D_H u|^2 -\a u$. For $n\geq 2$, in $\Omega\smallsetminus \{|x|=0\}$ we have
{\allowdisplaybreaks
\begin{align}\label{magik}
\frac{\Fa}{16} \LH v &= 2\left(\|M\|^2-\frac{1}{n+1}({{\rm trace}}(M))^2  \right)+\frac{2n+\a}{n^2(n+1)}\left(\sum_{j=1}^n U_j - \frac{n}{2}\Fa\right)^2 \\
&+ \frac{(4-\a)(2n+\a)(n-1)}{4n+4}\frac{|x|^4}{\rho^8}\left(\sum_{j=1}^n (x^2_j+x^2_{n+j}) U_j + t U_t - \frac{1}{2}\rho^\a\right)^2\notag\\
&+ \frac{(2n+\a)(n-2)}{2n(n+1)} \sum_{j=1}^n \left(1-n\frac{x^2_j+x^2_{n+j}}{|x|^2}\right)\left(U_j-\frac{1}{2}\Fa\right)^2   \notag\\
&+ \frac{4-\a}{4} \sum_{j=1}^n \left(\frac{(n-1)(2n+4)}{n+1} \frac{|x|^2(x^2_j+x^2_{n+j})}{\rho^8} + \frac{4-2n}{n+1}\frac{x^2_j+x^2_{n+j}}{|x|^2\rho^4}\right)\left(|x|^2U_t - t U_j\right)^2   \notag\\
&+\frac{4-\a}{8(n^2-1)}\frac{1}{\rho^4}\sum_{i\neq j}\left(t U_i+tU_j-2|x|^2U_t\right)^2\notag\\
&+\sum_{i\neq j}\left(U_i-U_j\right)^2\left[\frac{(4-\a)(n-2)^2}{8n^2(n^2-1)} + \frac{8+4n-n^2}{4n^2(n+1)}+\frac{3n-6}{4n+4}\frac{(x^2_i+x^2_{n+i})(x^2_j+x^2_{n+j})}{|x|^4}\right.\notag\\
&\hspace{2cm}- \frac{3}{4n+4}\frac{x^2_i+x^2_{n+i}+x^2_j+x^2_{n+j}}{|x|^2} + \frac{4-\a}{4n+4}\frac{|x|^2(x^2_i+x^2_{n+i}+x^2_j+x^2_{n+j})}{\rho^4}\notag\\
&\hspace{3cm}+\frac{(4-\a)(4-2n)}{4n+4}\frac{(x^2_i+x^2_{n+i})(x^2_j+x^2_{n+j})}{\rho^4} - \frac{4-\a}{8(n^2-1)}\frac{|x|^4}{\rho^4}\notag\\
&\hspace{3cm}+\left. \frac{(4-\a)(n-1)(n+2)}{4n+4}\frac{|x|^4(x^2_i+x^2_{n+i})(x^2_j+x^2_{n+j})}{\rho^8}\right].\notag
\end{align}
}
Concerning the case $n= 1$, in $\Omega\smallsetminus \{|x|=0\}$ we have
\begin{align}\label{magikuno}
\frac{\Fa}{16} \LH v &= 2\left(\|M\|^2-\frac{1}{2}({{\rm trace}}(M))^2  \right)+\frac{2+\a}{2}\left(U_1 - \frac{1}{2}\Fa \right)^2 \\
&+ \frac{4-\a}{2}\frac{1}{\rho^4}\left(|x|^2 U_t - t U_1\right)^2.   \notag
\end{align}
The $(n+1)\times(n+1)$ symmetric matrix $M$ appearing in both \eqref{magik} and \eqref{magikuno} is defined in \eqref{defmatrix} below, and it involves first and second derivatives of $U$.
\end{lemma}
\begin{proof}
Let us fix some notations: for $n\geq 1$ and $x\in\R^{2n}$ we denote
$$
s=(s_1,\ldots,s_n) \quad \mbox{ with }\quad s_j=x^2_j+x^2_{n+j}\mbox{ for }j\in\{1,\ldots,n\}
$$
and
$$
\sigma=\sum_{j=1}^{n}s_j=|x|^2.
$$
With this choice for the change of variables and recalling the assumption $u(x,t)=U(s,t)$, we have
$$
\frac{1}{4}|D_H u|^2=\sigma U^2_t + \sum_{j=1}^n s_j U^2_j
$$
and
$$
\frac{1}{4}\LH u=\sigma U_{tt} + \sum_{j=1}^n \left(s_j U_{jj}+ U_{j}\right)=:\mathcal{L} U
$$
where $U_j$ stands for $U_{s_j}$ and similar notations hold for the second derivatives. If we keep denoting by $\Fa$ and $v$ these functions after the change of variables, it yields
$$
\Fa=\sigma (\sigma^2+t^2)^{\frac{\a-4}{4}}
$$
and
$$
v=4\frac{\sigma U^2_t + \sum_{j=1}^n s_j U^2_j}{\sigma (\sigma^2+t^2)^{\frac{\a-4}{4}}}-\a U
$$
for any $\sigma>0$. With the notation
$$
g(s,t)=\frac{\sigma U^2_t(s,t) + \sum_{j=1}^n s_j U^2_j(s,t)}{\sigma (\sigma^2+t^2)^{\frac{\a-4}{4}}},
$$
we get
\begin{equation}\label{fromhere}
\frac{\Fa}{16} \LH v=\sigma (\sigma^2+t^2)^{\frac{\a-4}{4}}\mathcal{L}g-\frac{\a}{4}\sigma (\sigma^2+t^2)^{\frac{\a-4}{4}}\mathcal{L} U.
\end{equation}
For $\sigma >0$ let us now introduce the matrix
\begin{equation}\label{defmatrix}
M= D_2 - D_1
\end{equation}
where
$$
D_1= E_1 + \frac{2\sigma}{\sigma^2+t^2} \frac{\a-4}{4} E_2
$$
and the $(n+1)\times(n+1)$ symmetric matrices $D_2$, $E_1$, and $E_2$ are defined by
$$
\left(D_2\right)_{ij}=\begin{cases}
\sqrt{s_i s_j}\, U_{ij} \quad\,\,\,\text{ if } i,j\in \{1,\ldots, n\},  \\
\sqrt{s_i \sigma}\, U_{it} \quad\,\,\,\,\,\text{ if } i\in \{1,\ldots, n\}\mbox{ and }j=n+1,  \\
\sigma U_{tt}\quad\quad\quad\,\,\,\,\text{ if } i=j=n+1,
\end{cases}
$$
$$
\left(E_1\right)_{ij}=\begin{cases}
\frac{\sqrt{s_i s_j}}{\sigma}\frac{U_i+U_j}{2}-\frac{1}{2}U_i\delta_{ij} \quad\,\,\,\text{ if } i,j\in \{1,\ldots, n\},  \\
0 \qquad\qquad\qquad\quad\quad\quad\,\,\,\,\,\text{ if } i\in \{1,\ldots, n\}\mbox{ and }j=n+1,  \\
\frac{\sum_{j=1}^n{s_jU_j}}{2\sigma}\quad\quad\quad\quad\quad\quad\,\,\text{ if } i=j=n+1,
\end{cases}
$$
$$
\left(E_2\right)_{ij}=\begin{cases}
\sqrt{s_i s_j}\, \frac{U_i+U_j}{2} \quad\,\,\,\text{ if } i,j\in \{1,\ldots, n\},  \\
\sqrt{\frac{s_i}{\sigma}}\, \frac{\sigma U_t+ t U_i}{2} \quad\,\,\,\text{ if } i\in \{1,\ldots, n\}\mbox{ and }j=n+1,  \\
 t U_{t}\qquad\quad\quad\quad\,\,\,\text{ if } i=j=n+1.
\end{cases}
$$
A straightforward (yet very painful) computation shows that
{\allowdisplaybreaks
\begin{align*}
\sigma (\sigma^2+t^2)^{\frac{\a-4}{4}}\mathcal{L}g&= 2\|D_2-D_1\|^2 - 2\|D_1\|^2 + 2 \sigma U_t \de_t\left(\mathcal{L}U\right) + 2\sum_{j=1}^n s_j U_j \de_{s_j}\left(\mathcal{L}U\right)\\
&+|\nabla_s U|^2+nU_t^2 - (n+2) \left(\sigma U^2_t + \sum_{j=1}^n s_j U^2_j\right) \left(\frac{1}{\sigma}-\frac{4-\a}{4}\frac{2\sigma}{\sigma^2+t^2}\right)\\
&+\sigma \left(\sigma U^2_t + \sum_{j=1}^n s_j U^2_j\right) \left(\frac{2}{\sigma^2}-\frac{(4-\a)\a}{16}\frac{4\sigma^2}{(\sigma^2+t^2)^2}-\frac{(4-\a)\a}{16}\frac{4t^2}{(\sigma^2+t^2)^2}\right).
\end{align*}
Plugging the previous identity in \eqref{fromhere}, keeping in mind the definition of the previously defined matrices, and using $\LH u=(2n+\a)\Fa$ outside of $\{|x|=0\}$, we obtain
\begin{align}\label{daqui}
\frac{\Fa}{16} \LH v&= 2\left(\|M\|^2-\frac{1}{n+1}({{\rm trace}}(M))^2  \right) + \frac{2}{n+1}({{\rm trace}}(M))^2 \\
&- 2\left\|E_1 + \frac{2\sigma}{\sigma^2+t^2} \frac{\a-4}{4} E_2\right\|^2\notag\\
&+ 2 \sigma U_t \de_t\left(\frac{2n+\a}{4}\sigma (\sigma^2+t^2)^{\frac{\a-4}{4}}\right) + 2\sum_{j=1}^n s_j U_j \de_{s_j}\left(\frac{2n+\a}{4}\sigma (\sigma^2+t^2)^{\frac{\a-4}{4}}\right)\notag\\
&+|\nabla_s U|^2+nU_t^2 - (n+2) \left(\sigma U^2_t + \sum_{j=1}^n s_j U^2_j\right) \left(\frac{1}{\sigma}-\frac{4-\a}{4}\frac{2\sigma}{\sigma^2+t^2}\right)\notag\\
&+\sigma \left(\sigma U^2_t + \sum_{j=1}^n s_j U^2_j\right) \left(\frac{2}{\sigma^2}-\frac{(4-\a)\a}{16}\frac{4\sigma^2}{(\sigma^2+t^2)^2}-\frac{(4-\a)\a}{16}\frac{4t^2}{(\sigma^2+t^2)^2}\right)\notag\\
&-\frac{\a(2n+\a)}{16}\sigma^2 (\sigma^2+t^2)^{\frac{2\a-8}{4}}.\notag
\end{align}
}
If we now substitute into \eqref{daqui} the following two identities
\begin{align}\label{traceformula}
&{{\rm trace}}(M)={{\rm trace}}(D_2)-{{\rm trace}}(E_1)-\frac{2\sigma}{\sigma^2+t^2} \frac{\a-4}{4}{{\rm trace}}(E_2)\\
&=\mathcal{L}U+\sum_{j=1}^n U_j\left(-\frac{1}{2}-\frac{3}{2}\frac{s_j}{\sigma}-\frac{2\sigma s_j}{\sigma^2+t^2} \frac{\a-4}{4}\right) 
-\frac{2\sigma t}{\sigma^2+t^2} \frac{\a-4}{4}U_t\notag\\
&=\frac{2n+\a}{4}\sigma (\sigma^2+t^2)^{\frac{\a-4}{4}}+\frac{1}{2}\sum_{j=1}^n U_j\left( 2s_j\left( \frac{4-\a}{4} \frac{2\sigma}{\sigma^2+t^2}-\frac{3}{2\sigma}\right)   -1\right)+\frac{4-\a}{4}\frac{2\sigma t}{\sigma^2+t^2} U_t,\notag
\end{align}
{\allowdisplaybreaks
\begin{align*}
&\left\|E_1 + \frac{2\sigma}{\sigma^2+t^2} \frac{\a-4}{4} E_2\right\|^2\\
&=\|E_1\|^2 +  \frac{4\sigma^2}{(\sigma^2+t^2)^2} \left(\frac{4-\a}{4}\right)^2\|E_2\|^2 - \frac{4-\a}{4} \frac{4\sigma}{\sigma^2+t^2}{{\rm trace}}(E_1 E_2)\\
&=\sum_{i,j=1}^{n} \left(\frac{\sqrt{s_i s_j}}{\sigma}\frac{U_i+U_j}{2}-\frac{1}{2}U_i\delta_{ij}\right)^2 + \frac{\left(\sum_{j=1}^n s_jU_j\right)^2}{4\sigma^2}\\
&+ \frac{4\sigma^2}{(\sigma^2+t^2)^2} \left(\frac{4-\a}{4}\right)^2 \left[\sum_{i,j=1}^{n} \frac{s_i s_j}{4} \left(U_i+U_j\right)^2 + \frac{1}{2} \sum_{j=1}^n\frac{s_j}{\sigma}\left(\sigma U_t + t U_j\right)^2 + t^2 U_t^2\right]\\
&-\frac{4-\a}{4} \frac{4\sigma}{\sigma^2+t^2}\left[ \sum_{i,j=1}^{n}\left(\frac{s_i s_j}{4\sigma} \left(U_i+U_j\right)^2-\frac{\sqrt{s_i s_j}}{4}\delta_{ij} U_i(U_i+U_j) \right) + \frac{t}{2\sigma} U_t\sum_{j=1}^n s_jU_j \right],
\end{align*}
}
and we keep track of all the constants involved, a long (yet very direct) computation shows that
{\allowdisplaybreaks
\begin{align*}
\frac{\Fa}{16} \LH v &= 2\left(\|M\|^2-\frac{1}{n+1}({{\rm trace}}(M))^2  \right)+\frac{2n+\a}{n^2(n+1)}\left(\sum_{j=1}^n U_j - \frac{n}{2}\sigma (\sigma^2+t^2)^{\frac{\a-4}{4}}\right)^2 \\
&+ \frac{(4-\a)(2n+\a)(n-1)}{4n+4}\frac{\sigma^2}{(\sigma^2+t^2)^2}\left(\sum_{j=1}^n s_j U_j + t U_t - \frac{1}{2}(\sigma^2+t^2)^{\frac{\a}{4}}\right)^2\\
&+ \frac{(2n+\a)(n-2)}{2n(n+1)} \sum_{j=1}^n \left(1-n\frac{s_j}{\sigma}\right)\left(U_j-\frac{1}{2}\sigma (\sigma^2+t^2)^{\frac{\a-4}{4}}\right)^2   \\
&+ \frac{4-\a}{4} \sum_{j=1}^n \left(\frac{(n-1)(2n+4)}{n+1} \frac{\sigma s_j}{(\sigma^2+t^2)^2} + \frac{4-2n}{n+1}\frac{s_j}{\sigma(\sigma^2+t^2)}\right)\left(\sigma U_t - t U_j\right)^2   \\
&+\frac{4-\a}{8(n^2-1)}\frac{1}{\sigma^2+t^2}\sum_{i\neq j}\left(t U_i+tU_j-2\sigma U_t\right)^2\\
&+\sum_{i\neq j}\left(U_i-U_j\right)^2\left[\frac{(4-\a)(n-2)^2}{8n^2(n^2-1)} + \frac{8+4n-n^2}{4n^2(n+1)}+\frac{3n-6}{4n+4}\frac{s_is_j}{\sigma^2}\right.\\
&\hspace{3cm}- \frac{3}{4n+4}\frac{s_i+s_j}{\sigma} + \frac{4-\a}{4n+4}\frac{\sigma (s_i+s_j)}{\sigma^2+t^2}+\frac{(4-\a)(4-2n)}{4n+4}\frac{s_is_j}{\sigma^2+t^2}\\
&\hspace{3cm}\left.- \frac{4-\a}{8(n^2-1)}\frac{\sigma^2}{\sigma^2+t^2}+ \frac{(4-\a)(n-1)(n+2)}{4n+4}\frac{\sigma^2 s_is_j}{(\sigma^2+t^2)^2}\right].
\end{align*}
}
The previous identity coincides with \eqref{magik} once we recall our fixed notations. We stress explicitly that the previous formula makes sense for $n\geq 2$ (whereas for $n=1$ the last terms involving a summation over the indices $\{i\neq j\}$ lose their meaning). When $n=1$, the same computation leading to the previous formula gives the following substitute formula
\begin{align*}
\frac{\Fa}{16} \LH v &= 2\left(\|M\|^2-\frac{1}{2}({{\rm trace}}(M))^2  \right)+\frac{2+\a}{2}\left(U_1 - \frac{1}{2}\sigma (\sigma^2+t^2)^{\frac{\a-4}{4}}\right)^2 \\
&+ \frac{4-\a}{2}\frac{1}{\sigma^2+t^2}\left(\sigma U_t - t U_1\right)^2   \notag
\end{align*}
which coincides with \eqref{magikuno}. The proof is thus complete.
\end{proof}

The right hand side in \eqref{magik} displays a long list of \emph{sum of squares}: a careful check of the constants involved leads to the desired sub-harmonicity of the function $v$ for $0<\alpha\leq 4$ at least in the cases of cylindrically symmetric sets in $\Hn$ and toric symmetric sets in $\H^2$.

\begin{corollary}\label{corricorri}
Fix $c>0$ and $0<\alpha\leq 4$. Let $u$ be a solution of \eqref{eq: overdet} in $\Omega$ according to \eqref{competitor}-\eqref{solution}. Then the function $v$ in \eqref{Pfunction} is $\LH$-subharmonic in $\Omega\smallsetminus \{|x|=0\}$ in the following cases
\begin{itemize}
\item[(i)] $\Omega$ is cylindrically symmetric, for any $n\geq 1$;
\item[(ii)] $\Omega$ is toric symmetric, for $n\in\{1,2\}$.
\end{itemize}
\end{corollary}
\begin{proof}
We first notice that the symmetries of the domain $\Omega$ (either toric or cylindrical) are inherited by the solution $u$ of \eqref{eq: overdet}. As a matter of fact, if $\Omega$ is toric symmetric and $(x,t)\in\Omega$ (with $|x|\neq 0$), one can consider any $(x',t)$ with the property that $(x')^2_k+(x')^2_{n+k}=x^2_k+x^2_{n+k}$ for each $k\in\{1,\ldots,n\}$ and one can look at the orthogonal transformation from $\R^{2n}$ onto $\R^{2n}$ such that it acts blockwise and the $k$-th block consists of the rotation sending $(x_k,x_{n+k})\mapsto (x'_k,x'_{n+k})$. Since such linear transformation commutes with the matrix $J$ in \eqref{defJ} and preserves the norms in $\R^{2n}$, then it commutes with the subLaplacian $\LH$ (see \eqref{explicitDelta}). Hence, the composition of the solution $u$ with this rotation solves the same problem \eqref{eq: overdet} (since $\Omega$ is toric symmetric, and the source and Neumann data depend on $x$ as a function of $|x|$): the uniqueness of the solutions to \eqref{weak} implies that $u(x,t)=u(x',t)$ and yields the toric symmetry of the solution. The same holds for cylindrically symmetric sets $\Omega$, as for $(x,t)\in\Omega$ (with $|x|\neq 0$) and any $(x',t)$ with $|x|=|x'|$ one can find an orthogonal transformation from $\R^{2n}$ onto $\R^{2n}$ which commutes with $J$ and sends $x\mapsto x'$ (the same argument applies and also in this case $u(x,t)=u(x',t)$ from the uniqueness of the solutions). In both scenarios we can then consider symmetric solutions and apply Lemma \ref{lemmalemma}.\\
We first deal with case $(ii)$. Since $\Omega$ is toric symmetric, from the smoothness of $u$ outside of the $t$-axis we can write $u(x,t)=U(s,t)$ for some function $U$ which is smooth outside of $\{s=0\}$ (we use the notations settled in the proof of Lemma \ref{lemmalemma}). If $n=1$ we have the formula \eqref{magikuno}. For $n=2$ (so $s=(s_1,s_2)$), by formula \eqref{magik} we have instead
{\allowdisplaybreaks
\begin{align}\label{tordue}
\frac{\Fa}{16} \LH v &= 2\left(\|M\|^2-\frac{1}{3}({{\rm trace}}(M))^2  \right)+\frac{4+\a}{12}\left(U_1 + U_2 - \Fa\right)^2 \\
&+ \frac{(4-\a)(4+\a)}{12}\frac{|x|^4}{\rho^8}\left(\sum_{j=1}^2 (x^2_j+x^2_{2+j}) U_j + t U_t - \frac{1}{2}\rho^\a\right)^2\notag\\
&+ \frac{4-\a}{4} \sum_{j=1}^2 \frac{8|x|^2(x^2_j+x^2_{2+j})}{3\rho^8}\left(|x|^2U_t - t U_j\right)^2 +\frac{4-\a}{12\rho^4}\left(t U_1+tU_2-2|x|^2U_t\right)^2\notag\\
&+2\left(U_1-U_2\right)^2\left[\frac{4-\a}{24}\frac{|x|^4}{\rho^4}+\frac{4-\a}{3}\frac{|x|^4(x^2_1+x^2_{2+1})(x^2_2+x^2_{2+2})}{\rho^8}\right].\notag
\end{align}
}
Keeping in mind that
\begin{equation}\label{matrixineq}
\|M\|^2\geq \frac{1}{n+1}({{\rm trace}}(M))^2 \quad\mbox{ for any $(n+1)\times (n+1)$ symmetric matrix $M$},
\end{equation}
it is clear that the right hand sides in \eqref{magikuno} and \eqref{tordue} are nonnegative when $\alpha$ lies in the interval $(0,4]$.\\
We not turn the attention to the case $(i)$, and we fix $n\geq 1$. Since $\Omega\subset \Hn$ is cylindrically symmetric, from the smoothness of $u$ outside of the $t$-axis we can write $u(x,t)=W(\sigma,t)$ for some function $W$ which is smooth in $\{\sigma>0\}$ (as in the proof of Lemma \ref{lemmalemma} we use the notation $\sigma=|x|^2$). We can then exploit \eqref{magik} once we have in mind that
$$
\sigma=\sum_{j=1}^{n}s_j\qquad\mbox{and}\qquad W_\sigma=U_i\mbox{ for each }i\in\{1,\ldots,n\}
$$
where $U(s,t)=W(\sigma,t)$. Plugging this information in \eqref{magik} and performing a straightforward computation, we obtain
\begin{align}\label{cyln}
\frac{\Fa}{16} \LH v &= 2\left(\|M\|^2-\frac{1}{n+1}({{\rm trace}}(M))^2  \right)+\frac{2n+\a}{n+1}\left(W_\sigma - \frac{1}{2}\Fa\right)^2 \\
&+ \frac{(4-\a)(2n+\a)(n-1)}{4n+4}\frac{|x|^4}{\rho^8}\left(|x|^2 W_\sigma + t W_t - \frac{1}{2}\rho^\a\right)^2\notag\\
&+ \frac{4-\a}{4\rho^4}\left(|x|^2W_t - t W_\sigma\right)^2  \left(\frac{(n-1)(2n+4)}{n+1} \frac{|x|^4}{\rho^4} + \frac{4}{n+1}\right).  \notag
\end{align}
We remark that \eqref{cyln} coincides (as it should) with \eqref{magikuno} for $n=1$. Exactly as before, by \eqref{matrixineq}, the right-hand side of \eqref{cyln} is nonnegative when $\alpha\in (0,4]$.
\end{proof}

We now want to discuss the local behaviour of the function $v$ around the singular set provided by the $t$-axis.

\begin{lemma}\label{regularity}
Let $n\geq 1 $ and $\a>0$. Let $u$ be a weak solution to \eqref{weak} in a competitor set $\Omega$, and denote $v=\frac{1}{\Fa}|D_H u|^2 -\a u$. Then we have
\begin{itemize}
\item[i)] for $n\geq 2$, $v$ is locally in $L^1$ for any $\a\in (0,4]$;
\item[ii)] for $n\geq 2$, $\LH v$ is locally in $L^1$ for any $\a\in (\frac{3}{4},4]$;
\item[iii)] for $n\geq 1$, if $\O$ is toric symmetric, then $v$ is locally bounded for $\a\in (2,4]$.
\end{itemize}
\end{lemma}
\begin{proof}
Since $u$ is bounded and $\LH u\in L^p$ for any $p<\frac{Q}{2-\alpha}$ (if $0<\a<2$) or just any $p<\infty$ (if $\a\geq 2$), the statement concerns the behaviour of the function
$$
g=\frac{1}{\Fa}|D_H u|^2.
$$
We recall that, by the subelliptic counterpart of the classical Calderon-Zygmund theory, all the second derivatives of $u$ along the horizontal direction $X_j$'s as well as $Tu$ belong locally to $L^p$ for any $p<\frac{Q}{2-\alpha}$ (if $0<\a<2$, and any $p$ otherwise). Moreover, by the Sobolev embedding of the Folland-Stein spaces we then have
$$
|D_H u|\, \mbox{ belongs locally to } \begin{cases}
L^\infty \quad\,\,\,\text{ if } \alpha> 1,  \\
L^q \quad\,\,\,\text{ for any }q<\frac{Q}{1-\a}\mbox{ if }0<\alpha\leq 1.  \\
\end{cases}
$$
Hence, item $i)$ (i.e. $g\in L^1$) easily follows from H\"older's inequality once we notice that, for $n\geq 2$ the function $|x|^{-2}$ is locally in $L^{p}$ for $p\in (1,n)$ and
$$
\frac{1}{n}+\frac{2}{Q}<1.
$$
Let us now turn the attention on item $ii)$. To this aim, we should first compute
\begin{align*}
\LH g&= \frac{2}{\Fa}\left[ \left\|\frac{1}{2}\left(X_iX_j+X_jX_i\right)_{i,j}\right\|^2 + 4 (Tu)^2 \left\|J\right\|^2 + \left\langle D_H u , D_H \left(\LH u \right) \right\rangle + 8 \left\langle D_H u , J D_H T u \right\rangle \right]\\
&-\frac{4\sum_{j,k=1^n}X_k \Fa X_kX_ju X_j u}{\Fa^2} + \frac{|D_H u|^2}{\Fa^3}\left(2|D_H \Fa|^2 - \Fa\LH\Fa\right),
\end{align*}
where we exploited the non-commutativity relation $[X_j,X_i]=4J_{ij} T$. In order to check whether $\LH g$ is locally in $L^1$ we just use the previously mentioned integrability properties for the derivatives of $u$, together with the explicit expressions in \eqref{derfa}, and H\"older's inequality. We separately mention the presence of the term $D_H T u$ which is a third derivative in terms of horizontal derivatives and belongs locally to $L^p$ for any $p<\frac{Q}{3-\a}$ since this is the threshold for the integrability of $D_H \Fa$. We also mention that the condition $\a > \frac{3}{4}$ comes from the term $(\Fa)^{-2} X_k \Fa X_kX_ju X_j u $: as a matter of fact such term is locally bounded above by $|x|^{-3} |X_kX_j u||X_j u|$ which is integrable for $\a>\frac{3}{4}$ since
$$
\frac{3}{2n}+ \frac{2-\frac{3}{4}}{Q}+ \frac{1-\frac{3}{4}}{Q}\leq 1 \mbox{ for }n\geq 2.
$$
We are thus left with the proof of item $iii)$ which concerns the local boundedness of $g$ for $n\geq 1$. As in the proof of Corollary \ref{corricorri}, since $\O$ is toric symmetric, the function $u$ inherits such symmetry and we can write $u$ in terms of the function $U$. We know from (the notations of) Lemma \ref{lemmalemma} that the function $g$ is described by 
$$4\frac{\sigma U^2_t(s,t) + \sum_{j=1}^n s_j U^2_j(s,t)}{\sigma (\sigma^2+t^2)^{\frac{\a-4}{4}}}=4(\sigma^2+t^2)^{\frac{4-\a}{4}}\left(U^2_t(s,t)+\sum_{j=1}^n \frac{s_j}{\sigma} U^2_j(s,t)\right).$$
Since $u$ and $U$ are $C^\infty$-smooth away from $0$, we deduce that $g$ is locally bounded around any point in $\O\smallsetminus\{0\}$ if $\a\leq 4$. If we have $\a>2$ then we claim that the function $g$ is bounded also around $0$ (whenever $0\in\O$). As a matter of fact, for $\a>2$ the function $\Fa$ is H\"older continuous and by the subelliptic analogue of the classical interior Schauder theory we have that the horizontal second derivatives of $u$ are also H\"older continuous. Therefore $Tu=U_t$ and $U_j$ are both well-defined and bounded in $0$ (this can be seen from a Taylor expansion of $u$ around $0$, see e.g. \cite[Theorem 20.3.2]{BLU}).    
\end{proof}

\section{Main results}\label{sec4}

In this section we provide the details for Theorem \ref{Tuno} and Theorem \ref{Tdue}. We start by providing a more general version of Theorem \ref{Tuno} where we assume the cylindrical symmetry for the set $\O$ and some a-priori knowledge for the behaviour of $v$ around $0$.

\begin{theorem}\label{Tunoalpha}
Fix $n\geq 1$, $\alpha\in (0,4)$, and $c>0$. Let $\Omega\subset \Hn$ be a competitor set in the sense of \eqref{competitor}. Assume that $\Omega$ is cylindrically symmetric, and $0\in\Omega$. Assume also that at least one of the following condition holds true in a neighborhood of $0$: $\LH v\in L^1$ or $v$ bounded. If there exists a solution $u$ to \eqref{eq: overdet} in the sense of \eqref{solution}, then 
$$
\Omega=B_{R}(0)\qquad\mbox{with}\qquad R= c^{\frac{2}{\alpha}}
$$
and
$$
u(x,t)=\frac{\left(|x|^4+t^2\right)^{\frac{\alpha}{4}}- c^2}{\alpha}.
$$
\end{theorem}
\begin{proof}
As in the proof of Corollary \ref{corricorri} (we keep using the same notations fixed there), we can assume that $u(x,t)=W(\sigma, t)$. In such variables, the function $v$ in \eqref{Pfunction} takes the form
$$
4(\sigma^2+t^2)^{\frac{4-\alpha}{4}}(W_\sigma^2+W_t^2)-\a W.
$$
We notice that $v$ is $C^{\infty}(\O\smallsetminus\{0\})$ and it is smooth up to the boundary (since $0\in\Omega$). From the boundary data for $u$, we have that
\begin{equation}\label{constant}
v\equiv c^2 \qquad\mbox{ on }\de\O
\end{equation}
Let us assume first that $\LH v$ is locally in $L^1$ around $0$, which implies $\LH v\in L^1(\Omega)$. Thanks to this fact, and recalling that $u$ solves \eqref{eq: overdet}, we can obtain from \eqref{constant} and the integral identity \eqref{medianulla}
\begin{align*}
\int_\O (-u)\LH v &= \int_{\O} \left(v\LH u - u\LH v\right) - \int_\O v\LH u \\
&= \int_{\O} \left(v\LH u - u\LH v\right) - (Q+\a-2)\int_\O v\Fa \\
&= \int_{\O} \divH \left( v D_H u - u D_H v\right) - (Q+\a-2)\int_\O v\Fa\\
&= \int_{\de\O} v |D_H u| d\sigma_H - (Q+\a-2)\int_\O v\Fa \\
&= c^2 \int_{\de\O} |D_H u| d\sigma_H - (Q+\a-2)\int_\O v\Fa \\
&= c^2 \int_{\O} \LH u - (Q+\a-2)\int_\O v\Fa \\
&= (Q+\a-2)\left( c^2\int_\O \Fa - \int_\O v\Fa  \right) =0.
\end{align*}
On the other hand, since $u\in C(\overline{\O})$ and $u<0$ in $\O$ by \eqref{sign} and $\LH v$ is nonnegative almost everywhere by Corollary \ref{corricorri} (item $(i)$), we deduce that $\LH v$ needs to vanish almost everywhere. In particular we have
$$\LH v \equiv 0\qquad\mbox{ in }\O\smallsetminus \{|x|=0\}.$$
Let us now assume the second possibility: $v$ is locally bounded around $0$, which implies $v\in L^\infty(\O)$. Since $v$ is smooth outside the single point $0$, the inequality $\LH v \geq 0$ holds true in $\Omega\smallsetminus \{0\}$. Moreover, since $0\in\O$, the boundary of $\Omega\smallsetminus \{0\}$ is nothing bu $\de\O \cup \{0\}$. Hence, \eqref{constant} is saying that $v$ attains the boundary datum $c^2$ for any point of $\de\left( \Omega\smallsetminus \{0\} \right)$ except from $\{0\}$. Since $v$ is bounded and a set formed by a single point is a $\LH$-polar set, we can exploit the maximum principle in \cite[Theorem 11.2.7]{BLU} to infer that $v\leq c^2$ in $\Omega\smallsetminus \{0\}$. Inserting this information in the integral identity \eqref{medianulla}, and recalling that $\Fa$ is positive outside of $\{|x|=0\}$, we deduce that $v\equiv c^2$ in $\O\smallsetminus \{|x|=0\}$.\\
Under both the circumstances we considered, we have reached the conclusion
\begin{equation}\label{harmonik}
\LH v \equiv 0\qquad\mbox{ in }\O\smallsetminus \{|x|=0\}.
\end{equation}
We can now invoke \eqref{cyln}: recalling that the equality case in the matrix inequality \eqref{matrixineq} arises just for matrices which are multiple of the identity, and exploiting $\a<4$, if we substitute \eqref{harmonik} in \eqref{cyln} we infer
$$
\begin{cases}
M= \frac{{{\rm trace}}(M)}{n+1} \mathbb{I}_{n+1},  \\
W_\sigma = \frac{1}{2}\Fa,  \\
|x|^2W_t - t W_\sigma =0,
\end{cases}
$$
for any point in $\O$ with $\sigma>0$. Hence we have
$$\begin{cases}
W_\sigma(\sigma,t) = \frac{1}{2}\sigma (\sigma^2+t^2)^{\frac{\alpha-4}{4}},  \\
W_t(\sigma,t) =\frac{1}{2}t (\sigma^2+t^2)^{\frac{\alpha-4}{4}}.
\end{cases}$$
This implies the existence of a constant $k$ such that $W(\sigma,t)=\frac{1}{\a}(\sigma^2+t^2)^{\frac{\alpha}{4}} + k$, which means
$$
u(x,t)=\frac{1}{\a}(|x|^4+t^2)^{\frac{\alpha}{4}} + k.
$$
Since $u$ is continuous the previous identity holds true not only for $|x|>0$ but for any $(x,t)\in\overline{\Omega}$. It is then easy to conclude that $\O$ is a gauge-ball centered at the origin, and to complete the proof of the desired statement. 
\end{proof}

The next two results complete the proof of Theorem \ref{Tuno}.

\begin{corollary}\label{Cuno}
Fix either $n\geq 2$ and $\alpha\in (\frac{3}{4},4)$, or $n=1$ and $\a\in(2,4)$. Fix also $c>0$. Let $\Omega\subset \Hn$ be a competitor set in the sense of \eqref{competitor}. Assume that $\Omega$ is cylindrically symmetric, and $0\in\Omega$. If there exists a solution $u$ to \eqref{eq: overdet} in the sense of \eqref{solution}, then 
$$
\Omega=B_{R}(0)\qquad\mbox{with}\qquad R= c^{\frac{2}{\alpha}}
$$
and
$$
u(x,t)=\frac{\left(|x|^4+t^2\right)^{\frac{\alpha}{4}}- c^2}{\alpha}.
$$
\end{corollary}
\begin{proof}
The case $n\geq 2$ and $\alpha\in (\frac{3}{4},4)$ follows from Theorem \ref{Tunoalpha} and item $ii)$ in Lemma \ref{regularity}. The case $n = 1$ and $\alpha\in (2,4)$ follows from Theorem \ref{Tunoalpha} and item $iii)$ in Lemma \ref{regularity}.
\end{proof}

\begin{theorem}\label{Tuno4}
Fix $n\geq 1$, $\alpha=4$, and $c>0$. Let $\Omega\subset \Hn$ be a competitor set in the sense of \eqref{competitor}. Assume that $\Omega$ is cylindrically symmetric. If there exists a solution $u$ to \eqref{eq: overdet} in the sense of \eqref{solution}, then there exists $t_0\in\R$ such that 
$$
\Omega=B_{\sqrt{c}}((0,t_0))\qquad\mbox{and}\qquad u(x,t)=\frac{|x|^4+(t-t_0)^2- c^2}{4}.
$$
\end{theorem}
\begin{proof}
Since $\alpha=4$ we are in the simpler situation of $u$ being $C^\infty$-smooth in $\Omega$. In particular, exploiting the cylndrical symmetry, also the function $v$ is smooth in $\Omega$ as it coincides with $4(W_\sigma^2+W_t^2)-4 W$. We can then proceed verbatim as in the proof of Theorem \eqref{Tunoalpha} (both proofs provided there work smoothly) in order to reach the conclusion \eqref{harmonik}. Since $\alpha=4$, when we substitute \eqref{harmonik} in \eqref{cyln}, we infer the identities
\begin{equation}\label{casesdue}
\begin{cases}
M= \frac{{{\rm trace}}(M)}{n+1} \mathbb{I}_{n+1},  \\
W_\sigma = \frac{1}{2}F_4,  
\end{cases}
\end{equation}
for any point in $\O$ with $\sigma>0$. Since (see \eqref{traceformula})
$$
{{\rm trace}}(M)=\frac{2n+4}{4}\sigma-\frac{3+n}{2}W_\sigma \qquad\mbox{ for }\alpha=4,
$$
if we insert the information $W_\sigma(\sigma,t)=\frac{1}{2}F_4=\frac{1}{2}\sigma$ in the definition of $M$ we deduce from \eqref{casesdue} that
$$\begin{cases}
W_\sigma(\sigma,t) = \frac{1}{2}\sigma,  \\
W_{tt}(\sigma,t) =\frac{1}{2}.
\end{cases}$$
This implies the existence of two constants $t_0, k \in\R$ such that
$$
W(\sigma,t)=\frac{\sigma^2+(t-t_0)^2}{4}+k
$$
for all points $(\sigma,t)$ with $\sigma>0$ such that $(x,t)\in\Omega$. Similarly to the proof of Theorem \eqref{Tunoalpha}, this shows that $\Omega$ is a gauge-ball centered at $(0,t_0)$ and completes the proof.
\end{proof}

We now consider the case of toric symmetric sets in $\H^2$. In the following theorem we treat the full range $(0,4)$ for the parameter $\alpha$ under a local integrability requirement for $\LH v$.

\begin{theorem}\label{Tduealpha}
Fix $n= 2$, $\alpha\in (0,4)$, and $c>0$. Let $\Omega\subset \H^{2}$ be a competitor set in the sense of \eqref{competitor}. Assume that $\Omega$ is toric symmetric, and $0\in\Omega$. Assume also that $\LH v$ is locally in $L^1$ around the points $\Omega \cap \{|x|=0\}$. If there exists a solution $u$ to \eqref{eq: overdet} in the sense of \eqref{solution}, then 
$$
\Omega=B_{R}(0)\qquad\mbox{with}\qquad R= c^{\frac{2}{\alpha}}
$$
and
$$
u(x,t)=\frac{\left(|x|^4+t^2\right)^{\frac{\alpha}{4}}- c^2}{\alpha}.
$$
\end{theorem}
\begin{proof} 
As in the proof of Corollary \ref{corricorri}, we can assume that $u(x,t)=U(s, t)$ using our notations. Since $u$ is smooth outside of $0\in\O$, the same holds for $U$. As noted above, in such variables the function $v$ in \eqref{Pfunction} takes the form
$$
4(\sigma^2+t^2)^{\frac{4-\a}{4}}\left(U^2_t(s,t)+\sum_{j=1}^n \frac{s_j}{\sigma} U^2_j(s,t)\right)-\a U.
$$
In particular such function is smooth outside of the $t$-axis, it is bounded in a neighborhood of $\de\O$, and it is continuous up to the points of the boundary sitting outside of the $t$-axis where it attains the constant value $c^2$.\\
Moreover, we know from item $i)$ in Lemma \ref{regularity} that $v\in L^1_{{\rm{loc}}}(\O)$. Since by assumption also $\LH v \in L^1_{{\rm{loc}}}(\O)$ and we know from item $(ii)$ in Corollary \ref{corricorri} that $\LH v \geq 0$ holds in a pointwise sense in $\O\smallsetminus\{|x|=0\}$, then $\LH v \geq 0$ in $\O$ in the weak sense of distributions. By Theorem \cite[Theorem 8.2.15]{BLU} we have the existence of a function $\tilde{v}$ such that $\tilde{v}=v$ almost everywhere and $\tilde{v}$ is $\LH$-subharmonic (in the sense of potential theory, i.e. it is upper semi-continuous and it stays below its solid average in the sense of \eqref{mean}; see \cite[Chapter 8]{BLU}). The function $\tilde{v}$ coincides with $v$ where $v$ is continuous, and in particular $\tilde{v}$ attains with continuity the boundary datum $c^2$ at every point in $\de\Omega\smallsetminus\{|x|=0\}$. Furthermore, we notice that $\tilde{v}$ is bounded in $\Omega$: around interior points the boundedness is a consequence of being sub-average (and $v\in L^1$), whereas in a neighborhood of the boundary we can exploit the fact that $v$ is bounded in order to bound uniformly the solid average of $v$. Lastly, we stress that the set $\de\Omega\smallsetminus\{|x|=0\}$ is formed by a finite number of isolated points: as a matter of fact, if we could take a sequence of points in $\de\Omega\smallsetminus\{|x|=0\}$ we would have by compactness a point in $\de\Omega\smallsetminus\{|x|=0\}$ with $\de_t$ as a tangential direction and this is not possible since these points have to be characteristic for $\de\Omega$ and $\de_t$ has to be the normal direction at these points (recall that $u$ is a defining function \eqref{defining}, and at these points $|D_H u|=0$). We can then apply the maximum principle in \cite[Theorem 11.2.7]{BLU} (recall that countable sets are $\LH$-polar) to the function $\tilde{v}$, and we deduce that $\tilde{v}\leq c^2$ in $\Omega$.\\
On the other hand, we know from \eqref{medianulla} that
$$
\int_\Omega (\tilde{v}-c^2)\Fa=\int_\Omega (v-c^2)\Fa=0.
$$
As in the proof of Theorem \ref{Tunoalpha}, we deduce that $\tilde{v}\equiv v\equiv c^2$ in $\Omega\smallsetminus \{|x|=0\}$. In particular
\begin{equation}\label{harmonik2}
\LH v \equiv 0\qquad\mbox{ in }\O\smallsetminus \{|x|=0\}.
\end{equation}
We can then plug \eqref{harmonik2} in \eqref{tordue}, and (using $\alpha<4$) we conclude the validity of
\begin{equation}\label{casestre}
\begin{cases}
M= \frac{{{\rm trace}}(M)}{3} \mathbb{I}_{3},  \\
U_1+U_2 = \Fa,  \\
(x_1^2+x_{2+1}^2)U_1+(x_2^2+x_{2+2}^2)U_2 + tU_t= \frac{1}{2}\rho^\alpha\\
tU_1+t U_2 =2|x|^2 U_t,\\
U_1=U_2,
\end{cases}
\end{equation}
for the points with $s_1+s_2>0$. Since \eqref{casestre} implies that $U_1(s,t)=U_2(s,t)=\frac{s_1+s_2}{2}((s_1+s_2)^2+t^2)^{\frac{\alpha-4}{4}}$ and $U_t(s,t)=\frac{t}{2}((s_1+s_2)^2+t^2)^{\frac{\alpha-4}{4}}$, we can conclude as in Theorem \ref{Tunoalpha}.
\end{proof}

With the next two final results we finish the proof of Theorem \ref{Tdue}.

\begin{corollary}\label{Cdue}
Fix $n= 2$, $\alpha\in (\frac{3}{4},4)$, and $c>0$. Let $\Omega\subset \H^{2}$ be a competitor set in the sense of \eqref{competitor}. Assume that $\Omega$ is toric symmetric, and $0\in\Omega$. If there exists a solution $u$ to \eqref{eq: overdet} in the sense of \eqref{solution}, then 
$$
\Omega=B_{R}(0)\qquad\mbox{with}\qquad R= c^{\frac{2}{\alpha}}
$$
and
$$
u(x,t)=\frac{\left(|x|^4+t^2\right)^{\frac{\alpha}{4}}- c^2}{\alpha}.
$$
\end{corollary}
\begin{proof}
It follows by combining Theorem \ref{Tduealpha} with item $ii)$ in Lemma \ref{regularity}.
\end{proof}

\begin{theorem}\label{Tdue4}
Fix $n= 2$, $\alpha=4$, and $c>0$. Let $\Omega\subset \H^2$ be a competitor set in the sense of \eqref{competitor}. Assume that $\Omega$ is toric symmetric. If there exists a solution $u$ to \eqref{eq: overdet} in the sense of \eqref{solution}, then there exists $t_0\in\R$ such that 
$$
\Omega=B_{\sqrt{c}}((0,t_0))\qquad\mbox{and}\qquad u(x,t)=\frac{|x|^4+(t-t_0)^2- c^2}{4}.
$$
\end{theorem}
\begin{proof}
Verbatim proceeding as in the proof of Theorem \ref{Tduealpha}, we establish the validity of $\LH v \equiv 0$ in $\Omega \smallsetminus \{|x|=0\}$. Since $\alpha=4$, when we exploit the identity $\LH v=0$ in \eqref{tordue}, we obtain
\begin{equation}\label{casesquattro}
\begin{cases}
M= \frac{{{\rm trace}}(M)}{3} \mathbb{I}_{3},  \\
U_1+U_2 = F_4.  
\end{cases}
\end{equation}
If we now use $U_1(s,t)+U_2(s,t)=s_1+s_2$ in \eqref{traceformula}, we have
$$
{{\rm trace}}(M)=\frac{3}{2}\left(\sigma- \frac{s_1 U_1+s_2U_2}{\sigma}\right).
$$
Hence, the matrix equality in \eqref{casesquattro} reads as
\begin{align*}
&\begin{pmatrix}
s_1\left(U_{11}-\frac{1}{2}\right)+\frac{1}{2} U_1 & \sqrt{s_1s_2}\left(U_{12}-\frac{1}{2}\right) & \sqrt{s_1\sigma} U_{1t} \\
\sqrt{s_1s_2}\left(U_{12}-\frac{1}{2}\right) & s_2\left(U_{22}-\frac{1}{2}\right)+\frac{1}{2} U_2 & \sqrt{s_2\sigma} U_{2t} \\
\sqrt{s_1\sigma} U_{1t} & \sqrt{s_2\sigma} U_{2t} & \sigma U_{tt}-\frac{s_1 U_1+s_2 U_2}{2\sigma}
\end{pmatrix}\\
&= \left(\frac{\sigma}{2}- \frac{s_1 U_1+s_2U_2}{2\sigma}\right)\begin{pmatrix}
1 & 0 & 0 \\
0 & 1 & 0 \\
0 & 0 & 1
\end{pmatrix}.
\end{align*}
The combination of the previous matrix identity with the relation $U_1(s,t)+U_2(s,t)=s_1+s_2$ (and the smoothness of $U$) yields $U_{11}=U_{22}=U_{12}=\frac{1}{2}$ and we finally obtain the validity of
$$\begin{cases}
U_1(s,t) = \frac{s_1+s_2}{2},  \\
U_2(s,t) = \frac{s_1+s_2}{2}, \\
U_{tt}(s,t) =\frac{1}{2}.
\end{cases}$$
We can then finish the proof as in Theorem \ref{Tuno4}.
\end{proof}

\end{document}